\documentclass[12pt]{amsart}

\usepackage[T1]{fontenc}
\usepackage[utf8]{inputenc}

\usepackage{amsmath, amsthm, amssymb, mathtools}
\usepackage{xcolor, hyperref}
\hypersetup{
    colorlinks,
    linkcolor={red!50!black},
    citecolor={green!40!black},
    urlcolor={blue!80!black}
}
\usepackage[capitalize]{cleveref}
\usepackage{tikz-cd}

\newcommand{\kbar}{\overline{k}}
\newcommand{\Q}{\mathbb{Q}}
\newcommand{\Z}{\mathbb{Z}}
\newcommand{\frakp}{\mathfrak{p}}
\renewcommand{\phi}{\varphi}
\newcommand{\eps}{\varepsilon}
\newcommand{\frakq}{\mathfrak{q}}
\newcommand{\frakr}{\mathfrak{r}}
\newcommand{\frakc}{\mathfrak{c}}

\DeclareMathOperator{\Gal}{Gal}
\DeclareMathOperator{\End}{End}
\DeclareMathOperator{\Isog}{Isog_{\mathrm{pp}}}
\DeclareMathOperator{\Cl}{Cl}
\DeclareMathOperator{\cyc}{cyc}
\DeclareMathOperator{\tor}{tor}
\DeclareMathOperator{\inert}{inert}
\newcommand{\splitker}[1]{\ker(#1)_{\mathrm{split}}}

\newtheorem{prop}{Proposition}[section]
\newtheorem{thm}[prop]{Theorem}

\newtheorem{cor}[prop]{Corollary}
\theoremstyle{definition}
\newtheorem{defn}[prop]{Definition}
\newtheorem{ex}[prop]{Example}
\newtheorem{rem}[prop]{Remark}
\newtheorem{algorithm}[prop]{Algorithm}

\title[Spanning isogeny classes of PPAS's with RM]{Spanning isogeny
  classes of principally polarized abelian surfaces with RM}
\author{Jean Kieffer}
\date{\today}

\begin{document}

\maketitle

\begin{abstract}
  We describe how isogenies between principally polarized abelian surfaces
  (PPAS) with real multiplication (RM) can be decomposed into elementary
  isogeny types. This leads to a strategy for enumerating the isogeny class of
  a given PPAS with~RM. This note is part of an ongoing work with R.~van
  Bommel, S.~Chidambaram, and E.~Costa and is not indended for separate
  publication.
\end{abstract}

\section{Isogenies and real endomorphisms}

\subsection{Notation}

We keep the same setting as~\cite[§2]{g2classes}. $k$ is a field of
characteristic zero (e.g.~a number field). We identify group schemes over~$k$
with the set of their $\kbar$-points endowed with an action of~$\Gal(\kbar/k)$.

We only consider abelian varieties, polarizations, isogenies and endomorphisms
that are defined over~$k$. Two isogenies are declared isomorphic if they have
the same domain and differ by an isomorphism on their targets.

If~$(A,\lambda_A)$ is a principally polarized abelian variety (PPAV) over~$k$,
then the \emph{Rosati involution} on~$\End_k(A)$ is
$\dagger:\phi\mapsto \lambda_A^{-1}\phi^\vee \lambda_A$, where~$\phi^\vee$
denotes the dual of~$\phi$. Note that the Rosati involution depends on the
choice of~$\lambda_A$.

An endomorphism~$\beta$ of~$A$ is called \emph{symmetric}
if~$\beta^\dagger = \beta$. The roots of its characteristic polynomial (a
polynomial of degree $\dim A$) are then real numbers. We say that~$\beta$ is
\emph{totally positive} if these roots are positive. The set of symmetric
(resp.~symmetric and totally positive) endomorphisms is denoted by
$\End_k(A)^\dagger$
(resp.~$\End_k(A)^\dagger_{>0}$). If~$\beta\in \End_k(A)^\dagger_{>0}$, then
the \emph{$\beta$-torsion subgroup} $A[\beta]$ carries a non-degenerate Weil
pairing.

\subsection{Correspondence between isogenies and endomorphisms}

We recall \cite[Lemma 2.1]{g2classes}, which summarizes results of
Mumford~\cite{mumford}.

\begin{prop}
  \label{prop:mumford}
  Let~$(A,\lambda_A)$ be a PPAV over any field~$k$. Then there is a one-to-one
  correspondence between
  \begin{itemize}
  \item isomorphism classes of isogenies~$A\to A'$ where~$A'$ is also endowed
    with a principal polarization~$\lambda_{A'}$, and
  \item pairs $(\beta,G)$ where~$\beta\in \End(A)$ is symmetric and totally
    positive, and~$G$ is a subgroup of~$A[\beta]$ which is maximal isotropic
    with respect to the Weil pairing.
  \end{itemize}
  The correspondence is as follows. Given an isogeny $\phi:A\to A'$, we
  \begin{equation}
    \label{eq:mumford}
    \beta = \lambda_A^{-1} \circ \phi^\vee \circ \lambda_{A'} \circ \phi
  \end{equation}
  and~$G = \ker\phi$. Given a pair $(\beta,G)$, we let $\phi:A \to A/G$ be the
  quotient isogeny, and endow~$A' = A/G$ with the unique principal polarization
  $\lambda_{A'}$ such that \eqref{eq:mumford} holds.
\end{prop}

We have $\deg(\beta) = \deg(\phi)^2$. We say that $\phi$ is a
\emph{$\beta$-isogeny}. Note that with this terminology, a ``2-step
$\ell$-isogeny'' as in~\cite{g2classes} is an $\ell^2$-isogeny.

\subsection{Isogeny classes}

Before using \cref{prop:mumford} to enumerate isogeny classes, we must specify
exactly what classes we are looking at.

\begin{defn}
  Let~$(A,\lambda_A)$ be a PPAV over~$k$. We define
  \begin{displaymath}
    \Isog(A) = \{(A',\lambda_{A'})\text{ PPAV over } k: A' \text{ and } A \text{ are isogenous}\}/\sim.
  \end{displaymath}
  In other words, $\Isog(A)$ consists of all isomorphism classes of
  \emph{abelian varieties endowed with a principal polarization} that are
  isogenous to~$A$, without any a priori compatibility between isogenies and
  polarizations.
\end{defn}

\begin{ex}
  Assume that the only abelian variety over~$k$ that is isogenous to~$A$ (as a
  nonpolarized abelian variety) is~$A$ itself. Then determining $\Isog(A)$
  means determining the set of possible principal polarizations on~$A$ up to
  isomorphism.
\end{ex}

We would like to compute $\Isog(A)$ when~$k$ is a number field and~$A$ is a
given abelian surface with real multiplication. (It is a finite set by
Faltings~\cite{faltings}.) The purpose of this note is to present the first step
in such a computation: we explain how any isogeny between such abelian surfaces
can be decomposed into isogenies of ``elementary'' types, using
\cref{prop:mumford} as our main tool.

\subsection{Assumptions in the rest of this note}
\label{subsec:assumptions}

$F = \End_k(A)^\dagger\otimes\Q$ is either~$\Q\times\Q$ or a real quadratic
field. (In the former case,~$A$ is isogenous over~$k$ to a product of
non-isogenous elliptic curves.) $\Z_F$ is the maximal order of~$F$
(where $\Z_F = \Z\times \Z$ if $F = \Q\times\Q$) and $\Cl(F)$ (resp.~$\Cl^+(F)$) is the
class group (resp.~narrow class group) of~$\Z_F$ (both trivial if
$F = \Q\times\Q$).

The identification $F = \End_k(A)^\dagger\otimes\Q$ propagates through
isogenies: if~$\phi:A\to A'$ then there exists a unique bijection
$F\to\End_k(A')^\dagger\otimes\Q$ such that~$\phi$ is compatible with these RM
structures. This allows us to identify $F = \End_k(A')^\dagger\otimes\Q$ as well.

\begin{rem}
  Actually, the content of this note may also be useful if $\End_k(A)^\dagger$
  is more complicated than a quadratic field, i.e.~if the endomorphism algebra
  of~$A$ is QM, $M_2(\Q)$, or $M_2(\mathrm{CM})$: even in such a case, for
  every $\alpha \in \End_k(A)^\dagger_{>0}$, $\Q(\alpha)$ is either~$\Q$,
  $\Q\times\Q$ or a real quadratic fields, and the methods of this note apply
  to decompose an~$\alpha$-isogeny into more elementary ones.
\end{rem}

\subsection{Acknowledgements} I thank Raymond van Bommel for his comments on an
earlier version of this note. 

\section{Reduction to maximal RM}

If~$A$ has RM by a non-maximal order, then it is isogenous to another abelian
surface with real multiplication by~$\Z_F$.

\begin{prop}
  \label{prop:max-order}
  Assume that $A$ has RM by $\mathcal{O}\subset \Z_F$, and let~$c$ be the
  conductor of $\mathcal{O}$ in~$\Z_F$. Let~$c = \prod_{i=1}^r \ell_i$ be a
  decomposition of~$c$ as a product of prime numbers. Then there exists an
  isogeny $\phi: A\to B$ defined over~$k$ such that~$B$ has RM by $\Z_F$, and
  $\phi$ decomposes as $ \phi = \phi_1\circ\cdots \circ \phi_r$, where~$\phi_i$
  is an $\ell_i$-isogeny for every $1\leq i\leq r$.
\end{prop}

\begin{proof}
  This is ``well known''. We can always find $x\in \Z_F$ such that
  $\Z_F = \Z[x]$ and $\mathcal{O} = \Z[cx]$. One can take $G = A[c] \cap A[cx]$
  as the kernel of~$\phi$. Since~$G\simeq (\Z/c\Z)^2$ as an abstract group
  and~$G$ is maximal isotropic in $A[c]$, we can indeed decompose $\phi$ as
  claimed.
\end{proof}

\begin{cor}
  \label{cor:max-rm-decomposition}
  We can enumerate $\Isog(A)$ as follows:
  \begin{enumerate}
  \item Find $(B,\lambda_{B})\in \Isog(A)$ such that $B$ has RM
    by~$\Z_F$ as in \cref{prop:max-order};
  \item \label{step:max-rm} Starting from~$B$, enumerate all
    $(B',\lambda_{B'})\in \Isog(B)$ such that $B'$ also has RM by~$\Z_F$;
  \item From each $B'$ found in step~\eqref{step:max-rm}, look for sequences of
    $\ell$-isogenies where~$\ell\in\Z$ are primes (possibly distinct ones), and
    add the isogenous p.p.~abelian surfaces we find to $\Isog(A)$.
  \end{enumerate}
\end{cor}

We know how $\ell$-isogenies for $\ell\in \Z_{\geq 1}$ can be classified and
computed, see~\cite{g2classes}. (Only Dieulefait's criteria have to be modified
to account for the larger endomorphism algebra.) In the rest of this note, we
concentrate on step~\eqref{step:max-rm} in this procedure, and we add to the
assumptions in~§\ref{subsec:assumptions} that~$A$ has RM by~$\Z_F$.

\section{Hurwitz-Maass isogenies}

If~$A$ is a p.p.~abelian surface with RM by~$\Z_F$, then~$A$ always admits
rational isogenies of the following form. Let~$\frakp$ be an ideal in~$\Z_F$
such that $\frakp^2$ is trivial in~$\Cl^+(F)$, and let~$\alpha\in \Z_F$ be a
totally positive generator of~$\frakp^2$. (This includes the case:
$\frakp = \Z_F$ and~$\alpha$ is a totally positive unit.) Then the subgroup
$A[\frakp]$ is $k$-rational, and is maximal isotropic in $A[\alpha]$. By
\cref{prop:mumford}, the quotient abelian surface $B = A/A[\frakp]$ thus admits
a unique principal polarization~$\lambda_A$ for which
$\phi: (A,\lambda_A)\to (B,\lambda_B)$ is an $\alpha$-isogeny.

We refer to~$\phi$ as the \emph{Hurwitz--Maass isogeny} attached to
$(\frakp, \alpha)$. In the moduli interpretation, they correspond to the action
of the Hurwitz--Maass extension of the Hilbert modular group on the Hilbert
surface that classifies p.p.~abelian surfaces with RM by $\Z_F$
\cite[§I.4]{vdG}. Because $A[\frakp]$ is stable under $\Z_F$, $B$ still has RM
by $\Z_F$.

If~$\beta\in \Z_F\setminus\{0\}$, then replacing $\frakp$ by $\beta\frakp$ and
$\alpha$ by $\beta^2\alpha$ yields an isomorphic PPAS. Moreover, composing
Hurwitz--Maass isogenies corresponds to multiplying ideals and their
generators. Thus, the Hurwitz--Maass isogeny class of~$A$ can be exhausted by
considering only:
\begin{enumerate}
\item Fixed representatives $\frakp$ of ideal classes in $\Cl^+(F)$ that are
  squares of elements in $\Cl(F)$, and arbitrary generators $\alpha_{\frakp}$ of
  $\frakp^2$ fixed once and for all; and
\item Fixed representatives~$\eps$ of totally positive units in~$\Z_F$ modulo
  squares (for $\frakp = \Z_F$).
\end{enumerate}

\begin{rem}
  By genus theory, we in fact find an action of a finite abelian 2-group on the
  set of PPAS's with RM by~$\Z_F$ by means of rational isogenies. If
  $\#\Cl^+(F) = 1$ and the fundamental unit of~$\Z_F$ has negative norm, then
  this group is trivial and we don't have to worry about Hurwitz--Maass
  isogenies. This includes the cases $F = \Q(\sqrt{\Delta})$ for all prime
  fundamental discriminants $\Delta < 220$ and~$\Delta=8$, as well as
  $F = \Q\times\Q$.
\end{rem}

Not all isogenies are Hurwitz--Maass: in particular, $\beta$-isogenies when
$\beta\in \Z_F$ is a split prime are not, so the classification is not yet
complete.

\section{The full classification}

Let $f:A\to A'$ be an $\alpha$-isogeny between PPAS with RM by~$\Z_F$, where
$\alpha$ is a totally positive element of~$\Z_F$. Consider the factorization
\begin{displaymath}
  \alpha = \prod \frakp^{e(\frakp)}
\end{displaymath}
into prime ideals of~$\Z_F$.  By assumption, $\ker(f)$ is stable under~$\Z_F$,
so by the Chinese remainder theorem, we have a decomposition
\begin{displaymath}
  \ker(f) = \bigoplus_{\frakp} K_\frakp, \quad \text{where } K_\frakp = \ker(f)\cap A[\frakp^{e(p)}].
\end{displaymath}
Each subgroup $K_\frakp$ is maximal isotropic for the alternating pairing
induced on $A[\frakp^{e(\frakp)}]$ from $A[\alpha]$.

Fix an ideal~$\frakp$ appearing in this decomposition that is split in~$\Z_F$
(i.e.~$N_{F/\Q}(\frakp)$ is a prime~$p$). After choosing a convenient basis,
we can identify $A[\frakp^{e(\frakp)}]$ with $(\Z/p^{e(\frakp)}\Z)^2$, where
$p = N_{F/\Q}(\frakp)$ with its standard alternating pairing
\begin{displaymath}
  ((a,b), (c,d)) \mapsto ad - bc \pmod{p^{e(\frakp)}}.
\end{displaymath}
For each $0\leq j\leq e(\frakp)-1$, let $n_j\in \{0,1,2\}$ be the integer such
that $\#(p^j K_p \cap A[\frakp]) = p^{n_j}$. Since $K_\frakp$ is maximal
isotropic, we have the following constraints:
\begin{enumerate}
\item $n_j + n_{e(\frakp)-1-j}\leq 2$ for every~$j$,
\item $\sum_{j=0}^{e(\frakp)-1} n_j = e(\frakp)$,
\item The sequence $n_0,\ldots, n_{e(\frakp)-1}$ is nonincreasing.
\end{enumerate}
Therefore, there exists an integer $0\leq m(\frakp)\leq (e(\frakp)-1)/2$ such
that
\begin{displaymath}
  n_j =
  \begin{cases}
    2 &\text{if } j < m(\frakp),\\
    1 &\text{if } m(\frakp) \leq j < e(\frakp) - m(\frakp),\\
    0 &\text{if } j\geq e(\frakp) - m(\frakp).
  \end{cases}
\end{displaymath}

\begin{defn}
  \label{def:types}
  With the above notation, we say that
  \begin{itemize}
  \item the \emph{cyclic type} of~$f$ is $\cyc(f) = \prod_{\frakp \text{ split}} \frakp^{e(\frakp) - 2m(\frakp)}$,
  \item the \emph{torsion type} of~$f$ is $\tor(f) = \prod_{\frakp \text{ split}} \frakp^{m(\frakp)}$,
  \item the \emph{inert type} of~$f$ is
    $\inert(f) = \prod_{\frakp \text{ inert}} \frakp^{e(\frakp)}$.
  \end{itemize}
  We call $\bigoplus_{\frakp \text{ split}} K_\frakp$ the \emph{split
    subgroup} of~$\ker(f)$, denoted by $\splitker{f}$.

  Note that $A[\tor(f)]\subset \splitker{f}$, while $\splitker{f}/A[\tor(f)]$
  is a cyclic subgroup of $A[\cyc(f)\tor(f)]/A[\tor(f)]$, which explains the
  name \emph{torsion type} and \emph{cyclic type} respectively. Note also that
  both $\inert(f)$ and $\tor(f)^2\cyc(f)$ are trivial in~$\Cl^+(F)$.
\end{defn}

The inert type of an isogeny always admits a totally positive generator (in
fact, an element of~$\Z$). Thus, we can always decompose an isogeny to separate
the inert type.

\begin{prop}
  \label{prop:factor-inert}
  Let~$f:A\to A'$ be an $\alpha$-isogeny as above, and let $n\in \Z_{\geq 1}$
  be a generator of $\inert(f)$. Then we can decompose of~$f$ over~$k$ as
  \begin{displaymath}
    \begin{tikzcd}
      A \ar[r, "g"] & B \ar[r, "h"] & A',
    \end{tikzcd}\\
  \end{displaymath}
  where~$B$ is a PPAS with~RM by~$\Z_F$; $g$ is an~$\alpha/n$-isogeny; $h$ is
  an~$n$-isogeny; and~$f = h\circ g$. Thus $\inert(g)$ is the trivial ideal,
  $\inert(h) = (n)$, and both $\tor(h)$ and $\cyc(h)$ are trivial.
\end{prop}

\begin{proof}
  $\splitker{f}$ is maximal isotropic in $A[\alpha/n]$ and is stable
  under~$\Z_F$, so the quotient isogeny $g: A\to B = A/\splitker{f}$ fits in a
  diagram as claimed. (Note that by considering the inert part of $\ker(f)$, we
  could also write $f = g\circ h$ where $h$ is an $n$-isogeny and $g$ is an
  $\alpha/n$-isogeny.)
\end{proof}

In general, $\Cl^+(F)$ is not trivial, so separating the cyclic and torsion
types is not so straightforward. Nevertheless, we have:

\begin{prop}
  \label{prop:factor-cyclic}
  Let $f:A\to A'$ be an $\alpha$-isogeny with trivial inert type. Let~$\frakq$
  and $\frakr$ be divisors of $\cyc(f)$ and $\tor(f)$ respectively such that
  $\frakq \frakr^2$ is trivial in~$\Cl^+(F)$, and let~$\beta$ be a totally
  positive generator of that ideal. Then there exists a diagram defined
  over~$k$:
  \begin{displaymath}
    \begin{tikzcd}
      A \ar[r, "g"] & B \ar[r, "h"] & A'
    \end{tikzcd}
  \end{displaymath}
  where~$B$ is a PPAS with~RM by~$\Z_F$; $g$ is a~$\beta$-isogeny with cyclic
  and torsion types~$\frakq$ and~$\frakr$; $h$ is an~$\alpha/\beta$-isogeny
  with cyclic and torsion types~$\cyc(f)/\frakq$ and~$\tor(f)/\frakr$;
  and~$f = h\circ g$.
\end{prop}

\begin{proof}
  For a split prime~$\frakp$ appearing in the decomposition of~$\alpha$,
  let~$p$, $e(\frakp)$ and $m(\frakp)$ be as above. Let~$e'$ and~$e''$ be the
  exponents of~$\frakp$ in $\frakq\frakr$ and $\frakq$ respectively. We then
  write
  \begin{displaymath}
    K'_{\frakp} = A[\frakp^{e''}] + p^{e(\frakp) - m(\frakp) - e'} K_\frakp.
  \end{displaymath}
  The subgroup $K' = \bigoplus_{\frakp} K'_{\frakp}$ is maximal isotropic in
  $A[\beta]$ and $k$-rational. Let~$g$ be the quotient isogeny $A\to A/K'$: it
  has the correct torsion and cyclic types. There exists a unique isogeny $h$
  such that $f = h\circ g$, and we can also check that~$h$ has the correct
  torsion and cyclic types.
\end{proof}

\begin{cor}
  \label{cor:factor-cyclic}
  Let $f:A\to A'$ be an $\alpha$-isogeny with trivial inert type. Let~$\frakq$
  be a divisor of $\cyc(f)$ which is a square in~$\Cl^+(F)$, and let~$\frakc$
  be any ideal such that $\frakc^2 \frakq$ is trivial in~$\Cl^+(F)$. Choose a
  totally positive generator~$\beta$ of~$\frakq \frakc^2$. Let
  $n = N_{F/\Q}(\frakc)$. Then there exists a commutative diagram
  \begin{displaymath}
    \begin{tikzcd}
      A \ar[d, swap, "{[n]}"] \ar[r, "g"] & B \ar[r, "h"] & A'
      \\
      A \ar[rru, swap, bend right=10, "f"]
    \end{tikzcd}
  \end{displaymath}
  where~$B$ is a PPAS with RM by~$\Z_F$, $g$ is a $\beta$-isogeny with torsion
  and cyclic types $\frakc$ and~$\frakq$, and $h$ is an $\alpha/\beta$-isogeny
  with torsion and cyclic types $\tor(f) \overline{\frakc}$ and
  $\cyc(f)/\frakq$.
\end{cor}

\begin{proof}
  We can assume that $\frakc$ is a split ideal. Then $f\circ[n]$ has torsion
  type $\tor(f)\frakc \overline{\frakc}$ and cyclic type $\cyc(f)$. Now we
  apply \cref{prop:factor-cyclic} using that $\frakc^2 \frakq$ is trivial in
  $\Cl^+(F)$.
\end{proof}

This leads us to the following way of enumerating $\Isog(A)$.

\begin{algorithm}
  \label{algo:rm-class}
  \textbf{Input:} A PPAS $A/k$ with real multiplication by~$\Z_F$.

  \noindent \textbf{Output:} The set $\Isog(A)$.
  \begin{enumerate}
  \item Let $S$ be the Hurwitz--Maass isogeny class of~$A$.
  \item In this step, we consider a fixed element $B\in S$. (At first,
    $B=A$.) Let~$Q$ be a set of nontrivial ideals $\frakq$ of $\Z_F$ such that:
    \begin{enumerate}
    \item $\frakq$ is a square in $\Cl^+(F)$, and $\frakq$ is a product of
      split ideals;
    \item $A[\frakq]$ admits a maximal cyclic sugroup that is $k$-rational;
    \item Every nontrivial ideal $\frakq'$ satisfying (a) and (b) is divisible
      by some $\frakq\in Q$.
    \end{enumerate}
    For every $\frakq\in Q$,
    \begin{enumerate}
    \item Choose an ideal $\frakc$, prime to~$\frakq$, so that $\frakc^2\frakq$
      is trivial in~$\Cl^+(F)$, and let~$\beta$ be a totally positive generator
      of that ideal.
    \item For every maximal cyclic $k$-rational sugroup $G\subset A[\frakq]$,
      the subgroup $G' = A[\frakc] \oplus G$ is maximal isotropic in
      $A[\beta]$. Let~$B$ be the PPAS $A/G'$. If $B\in S$, then do nothing,
      otherwise add the Hurwitz--Maass isogeny class of~$B$ to~$S$ and apply
      step (2) to this new $B$.
    \end{enumerate}
  \item Apply step~(4) to every $B\in S$.
  \item In this step, we consider a fixed element $B\in S$. Let~$L_1$
    (resp.~$L_2$) be the set of primes $\ell\in \Z_{\geq 1}$ such that
    $A[\ell]$ (resp.~$A[\ell^2]$) admits a maximal isotropic $k$-rational
    subgroup isomorphic to $(\Z/\ell\Z)^2$
    (resp.~$(\Z/\ell\Z)^2\times (\Z/\ell^2\Z)$).
    \begin{enumerate}
    \item For every $\ell\in L_1$, add all the PPAS's $B'$ that are 1-step
      $\ell$-isogenous to~$B$ to~$S$.
    \item For
      every $\ell\in L_2$, add all the PPAS's $B'$ that are 2-step
      $\ell$-isogenous to~$B$ to~$S$.
    \item Apply step (4) to every element of~$S$ we
    added in (a) or~(b).
    \end{enumerate}
  \item Ouput $S = \Isog(A)$.
  \end{enumerate}
\end{algorithm}

For the notion of 1-step and 2-step $\ell$-isogenies where $\ell$ is a prime
number, see~\cite{g2classes}.

\begin{thm}
  \Cref{algo:rm-class} is correct.
\end{thm}

\begin{proof}
  First, the algorithm terminates, as $\Isog(A)$ is finite and step~(2) is
  applied at most once on each PPAS in this set.

  Let~$A'\in \Isog(A)$, and fix an isogeny $f:A\to A'$. By
  \cref{cor:max-rm-decomposition}, \cref{prop:factor-inert} and step (4), we
  can assume that $A'$ has RM by $\Z_F$ and~$f$ has trivial inert type. If
  $\cyc(f)$ is nontrivial, then when we apply step (2) to~$A$, we find
  $\frakq\in Q$ dividing $\cyc(f)$. Choosing $\frakc$ as in step~(2), by
  \cref{cor:factor-cyclic}, we find a decomposition $f\circ[n] = g\circ h$
  where $h$ has cyclic type $\cyc(f)/\frakq$, and moreover step (2) has been
  applied to the domain of~$h$. After a finite number of such steps, we reduce
  to the case where the cyclic type of~$f$ is trivial. Then $f$ is a
  Hurwitz--Maass isogeny. Since step~(2) adds the Hurwitz--Maass isogeny class
  of any PPAS it finds,~$A'$ will indeed be part of~$S$.
\end{proof}

In order to actually implement \cref{algo:rm-class}, we have to do
more. Computing $Q$ in step (2), and $L_1,L_2$ in step (4), will rely on
Dieulefait's criteria. Computing isogenies will rely on analytic methods as in
the typical case \cite{g2classes}.

\end{document}